\documentclass[11pt]{amsart}

\usepackage{amssymb}
\usepackage{amsaddr}

\newcommand{\RR}{\mathbb{R}}

\newcommand{\be}{\begin{equation}}
\newcommand{\ee}{\end{equation}}
\newcommand{\beqn}{\begin{eqnarray}}
\newcommand{\eeqn}{\end{eqnarray}}
\newcommand{\beqnno}{\begin{eqnarray*}}
\newcommand{\eeqnno}{\end{eqnarray*}}

\newcommand{\tv} {{\tilde v}}
\newcommand{\tu} {{\tilde u}}

\numberwithin{equation}{section}

\newtheorem{thm}{Theorem}[section]
\newtheorem{prop}[thm]{Proposition}
\newtheorem{cor}[thm]{Corollary}
\newtheorem{lem}[thm]{Lemma}

\newtheorem*{thm*}{Theorem}

\begin{document}

\title[Pointwise for Besel heat and Poisson initial data problems]{On the pointwise convergence to initial data of heat and Poisson problems for the Bessel operator}

\author{Isolda Cardoso}

\address{ Departamento de Matem\'{a}tica  \\
  Fac. de Cs. Exactas, Ingenier\'{i}a y Agrimensura  \\
 Universidad Nacional de Rosario \\
  Pellegrini 250, 2000 Rosario
 }
 \email{isolda@fceia.unr.edu.ar }

\thanks{The author is partially supported by  SeCyT, Universidad Nacional de Rosario and a grant from Consejo Nacional de Investigaciones Cient\'{i}ficas y T\'{e}cnicas (CONICET).}



\begin{abstract}
We find optimal integrability conditions on the initial data $f$ for the existence of solutions $e^{-t\Delta_{\lambda}}f(x)$ and $e^{-t\sqrt{\Delta_{\lambda}}}f(x)$ of the heat and Poisson initial data problems for the Bessel operator $\Delta_{\lambda}$ in $\RR^{+}$. We also characterize the most general class of weights $v$ for which the solutions converge a.e. to $f$ for every $f\in L^{p}(v)$, with $1\le p<\infty$. Finally, we show that for such weights and $1<p<\infty$ the local maximal operators are bounded from $L^{p}(v)$ to $L^{p}(u)$, for some weight $u$.
\end{abstract}

\maketitle

\section{Introduction}

\par The starting point for us are the classical heat and Poisson equations in the upper halfplane. For a nonnegative, second order differential operator $L$ we have the initial data problems
\begin{equation}\label{heat.eqn}\tag{\textbf{h}}
 \left\{ \begin{array}{rcll}
				 u_{t} \quad =& -Lu,&  t\in (0,T), \\
				  u(0,\cdot)\quad=&f(\cdot);&
				  \end{array} \right.
\end{equation}
\begin{equation}\label{Poisson.eqn}\tag{\textbf{P}}
  \left\{ \begin{array}{rcll}
				 u_{tt} \quad=& Lu, & t\in (0,\infty), \\
				  u(0,\cdot)\quad=&f(\cdot).&
				  \end{array} \right.
\end{equation}

\par In several classic examples it is well known that under certain conditions on the initial data $f$, for example when $L$ is the Laplacian on $\RR^{d}$ and $f\in L^{p}$ (or even some weighted $L^{p}(v)$), the solutions to the problems \textbf{\ref{heat.eqn}} and \textbf{\ref{Poisson.eqn}} verify $$\lim\limits_{t\to 0^{+}} u(t,x)=f(x).$$ Moreover, they can be described by the heat diffusion semigroup $u(t,x)=e^{-tL}f(x)$ and the Poisson semigroup $u(t,x)=e^{-t\sqrt{L}}f(x)$, respectively, with $L$ being its infinitesimal generator. In order to obtain a pointwise convergence of the solutions to the initial data it is necessary to study  the boundedness in $L^{p}$ of the corresponding maximal operators (see for example \cite{S}).

\par In the present notes we are interested in one particular differential operator: the Bessel operator in $\RR^{+}$, namely $\Delta_{\lambda}$, for $\lambda>-\frac{1}{2}$ as appears in the works of Muckenhoupt and Stein (see for example \cite{MS}).

\par Let us first consider the Bessel heat equation. We want to find optimal integrability conditions on the initial data $f$ such that
\begin{itemize}
\item[(I)] $u(t,x)=e^{-t\Delta_{\lambda}}f(x)$ exists for all $t,x>0$ as an absolutely convergent integral, satisfies the Bessel heat equation and
\item[(II)] $\lim\limits_{t\to 0^{+}} u(t,x)=f(x)$ for a.e. $x>0$.
\end{itemize}
The same is asked when we consider the Poisson equation for the Bessel operator, with $u(t,x)=e^{-t\sqrt{\Delta_{\lambda}}}f(x)$.

\par When $L$ is the Laplacian in $\RR^{d}$, Hartzstein, Torrea and Viviani in \cite{HTV} characterized the class of weights $D_{p}^{heat}(L)$ and $D_{p}^{Poisson}(L)$ for which the corresponding solutions have limits almost everywhere, respectively, for functions in $L^{p}(\RR^{+},vdx)$. Also, this problem is solved when $L$ is the Hermite or the Ornstein-Uhlenbeck operator in \cite{GHSTV} and in \cite{GHSTV1} for the Laguerre operator.

\par Here, in the Bessel setting for the heat and Poisson initial data problems, we also try to answer the natural question: \textit{Is there a weight class which characterizes convergence to initial data for functions in $L^{p}(\RR^{+},vdx)$?} This is, we want to find optimal conditions in a weight $v$ such that (I) and (II) hold for all $f\in L^{p}(v)$.

\par The final question that we want to adress is the $L^{p}(w)$ boundness of the corresponding local maximal operators, namely
\begin{align*}
W^{\lambda,\ast}_{a}f(x)=&\sup\limits_{0<t<a} |e^{-tL}f(x)|, \\
P_{a}^{\ast}f(x)=&\sup\limits_{0<t<a} |e^{-t\sqrt{L}}f(x)|.
\end{align*}
We want to show that for all weights $v$ in the class $D_{p}^{heat}(\Delta_{\lambda})$, the local maximal operator $W^{\lambda,\ast}_{a}$ maps $L^{p}(v)$ to $L^{p}(u)$, boundedly, for some weight $u$. Also, the analogous problem involving weights in $D_{p}^{Poisson}(\Delta_{\lambda})$ and the local maximal operator $P^{\lambda,\ast}_{a}$ will be treated.\\
In the Bessel context sharp power- weighted $L^p$ inequalities for the (global) heat and Poisson maximal operators are obtained in \cite{BHNV}.

\par In order to study the above problems we will follow the techniques developed in \cite{HTV}, \cite{GHSTV} and \cite{GHSTV1}.

\par More precisely, our main results are the following:

\begin{thm}\label{thm:Main}
Let $\lambda\geq0$,  $T>0$ (possibly $T=\infty$) and let $f:\RR^{+}\to\RR$ be a measurable function such that
	\begin{equation}
	\label{condition.Main}
		\int\limits_{0}^{\infty} \phi(y) |f(y)|dy<\infty,
	\end{equation}
where $\phi$ is the integrability factor $$\phi(y)=\phi_{t}^{\lambda}(y)= y^{\lambda}\left(\frac{y}{y+1}\right)^{\lambda} e^{-{\frac{y^{2}}{4t}}},\,\,\,\,\quad t\in(0,T), $$ for problem \textbf{(h)} or $$\phi(y)=\phi^{\lambda}(y)= {\frac{y^{2\lambda}}{(y^{2}+1)^{\lambda+1}}}$$ for problem \textbf{(P)}. Then the heat or Poisson integral $u(t,x)$ defines, respectively, an absolutely convergent integral such that
	\begin{itemize}
		\item[(i)] $u\in C^{\infty}((0,T)\times\RR^{+})$ and  satisfies the heat or Poisson equation for the Bessel operator $\Delta_{\lambda}$, and
		\item[(ii)] $\lim\limits_{t\to 0^{+}} u(t,x)=f(x)$ for a.e. $x>0$.
	\end{itemize}
Conversely, if a nonnegative function $f$ satisfies that its heat integral for $t\in(0,T)$ and some $x>0$ or its Poisson integral is finite for some $t\in(0,\infty)$ and some $x>0$ then $f$ must satisfy condition \ref{condition.Main}.
\end{thm}
\par From Theorem \ref{thm:Main} we see that the conditions on initial data $f$ for the existence of solutions cover a wide class of functions. For example, $f(y)= P(y)e^{\frac{1}{4T}y^{2}}$ for $y>1$ with $P$ any polynomial and $f(y)= y^{-2\lambda-\epsilon}$ if $y\leq1$ and $0<\epsilon<1$,  is good enough to grant existence of solutions to Problem \textbf{(h)}, and $f(y)=\frac{y}{\log^\beta(y+e)}$ for $y>1,  \beta >1$ and $f(y)= y^{-2\lambda-\epsilon}$ if $y\leq1$ with  $0<\epsilon<1$, is admissible for the
existence of the solution of  Problem \textbf{(P)} with $L=\triangle_{\lambda}$.
\par The classes \begin{equation}\label{B.weight.class} D^{heat}_{p}(\Delta_{\lambda}) \mbox{ and } D^{Poisson}_{p}(\Delta_{\lambda}) \end{equation} consist of all the weights $v:\RR^{+}\to\RR^{+}$ such that properties (i) and (ii) from Theorem \ref{thm:Main} hold for every function $f$ in the weighted space $L^{p}(\RR^{+},v)$.

\par From the theorem above we deduce a characterization of the classes $D^{heat}_{p}(\Delta_{\lambda})$ and $D^{Poisson}_{p}(\Delta_{\lambda})$.

\begin{cor}\label{cor:B.weighted.class.characterization}
Let $1\le p<\infty$, let $T>0$ (possibly $T=\infty$) and $\phi_{t}^{\lambda}$ as in Theorem \ref{thm:Main}. A weight $v$ belongs to the class $D_{p}^{heat}(\Delta_{\lambda})$ if and only if $v^{-{\frac{1}{p}}}\phi_{t}^{\lambda}\in L^{p'}(\RR^{+})$, and belongs to the class $D_{p}^{Poisson}(\Delta_{\lambda})$ if and only if $v^{-{\frac{1}{p}}}\phi^{\lambda}\in L^{p'}(\RR^{+})$.
\end{cor}

\par Just like in \cite{HTV}, we have that the weight classes are related by inclusion: $D^{Poisson}_{p}(\Delta_{\lambda})\subsetneq D^{heat}_{p}(\Delta_{\lambda})$. Indeed, let us consider a weight $v$ in the class $D^{Poisson}_{p}(\Delta_{\lambda})$. We have that $\phi_{t}^{\lambda}(y)=(y^{2}+1)^{\frac{\lambda}{2}+1}e^{-\frac{y^{2}}{4t}}\phi^{\lambda}(y)\le c_{t} \phi^{\lambda}(y)$, hence $||v^{-\frac{1}{p}}\phi_{t}^{\lambda}||_{p'}<\infty$. Also, the weight defined by $v(y)=e^{-\frac{y^{2}}{4T}p}$for $y>1$ and $v(y)=y^{2\lambda+\epsilon}$ if $y\leq1$ and $\epsilon<\frac{1}{p'}$, belongs to $D^{heat}_{p}(\Delta_{\lambda})$ and doesn't belong to $D^{Poisson}_{p}(\Delta_{\lambda})$.

\begin{thm}\label{thm:Lpv.Lpu.boundness.maximal.B.heat.operator}
Let $1<p<\infty$, $T>0$ (possibly $T=\infty$)  and $\lambda\geq 0$.
\par (i)  If $v\in D_{p}^{heat}(\Delta_{\lambda})$ then, for every $a\in (0,T)$ there exists a weight $u=u_{a}$ such that
\begin{equation}\label{B.maximal.operator.boundness}
W^{\lambda,\ast}_{a}:L^{p}(v)\to L^{p}(u) \qquad \mbox{ boundedly.}
\end{equation}
Moreover, there exists $\sigma_{0}=\sigma_{0}(a,T)\in (0,1)$ such that for any $\sigma\le\sigma_{0}$ the weight $u$ can be chosen such that also $u^{\sigma}\in D_{p}^{heat}(\Delta_{\lambda})$. (In the case that $T= \infty$ we can choose $u$ such  that $u^{\sigma}\in D_{p}^{heat}(\Delta_{\lambda})$ for all $\sigma < 1$).
\par Conversely, if \ref{B.maximal.operator.boundness} holds for some weight $u=u_{a}$ and each $a\in (0,T)$, then $v\in D_{p}^{heat}(\Delta_{\lambda})$.
\par (ii)  If $v\in D_{p}^{Poisson}(\Delta_{\lambda})$ then, for every $a>0$  there exists a weight $u=u_{a}$ such that
\begin{equation}\label{B.maximal.operator.boundness.P}
P^{\lambda,\ast}_{a}:L^{p}(v)\to L^{p}(u) \qquad \mbox{ boundedly.}
\end{equation}
If $\sigma<1$ we can find $u$ such that $u^{\sigma}\in D_{p}^{Poisson}(\Delta_{\lambda})$.
\par Conversely, if \ref{B.maximal.operator.boundness.P} holds for some $a>0$ and same weight $u=u_{a}$, then $v\in D_{p}^{Poisson}(\Delta_{\lambda})$.
\end{thm}

\par In section \ref{Preliminaries} we state all we need to recall about the Bessel operator for explicit computations. On sections \ref{heat} and \ref{Poisson} we study the problems stated above and prove Theorem \ref{thm:Main}, Corollary \ref{cor:B.weighted.class.characterization} and Theorem \ref{thm:Lpv.Lpu.boundness.maximal.B.heat.operator} for the heat and Poisson problems associated to the Bessel operator, respectively.

\section{Preliminaries}\label{Preliminaries}

\par Let us consider the Bessel operator as appears in \cite{MS}:
\begin{equation}\label{B.Operator}
			\Delta_{\lambda}=-{\frac{d^{2}}{dx^{2}}} - {\frac{2\lambda}{x}} {\frac{d}{dx}},
	\end{equation}
for $\lambda>-{\frac{1}{2}}$, which is essentially self-adjoint in $L^{2}(\mathbb{R}^{+},d\mu_{\lambda})$, where $\RR^{+}=(0,\infty)$ and $d\mu_{\lambda}(x)=x^{2\lambda}dx$, $x>0$.

\par Let us also consider the heat equation with initial data problem for the Bessel operator
									\begin{equation} \label{B.heat.eqn}
										\left\{	\begin{array}{rcll}
																u_{t}&=&-\Delta_{\lambda}u, & t>0,\\
																u(0,\cdot)&=&f(\cdot);&
														\end{array}
										\right.
									\end{equation}
and the Poisson equation with initial data problem for the Bessel operator
								\begin{equation} \label{B.Poisson.eqn}
										\left\{	\begin{array}{rcll}
																u_{tt}&=&\Delta_{\lambda}u, & t>0,\\
																u(0,\cdot)&=&f(\cdot);&
														\end{array}
										\right.
									\end{equation}

\par The standard set of eigenfunctions of the Bessel operator consists of
	\begin{equation}\label{B.Eigenfunctions}
		\varphi_{z}^{\lambda}(x)=(zx)^{-\lambda+{\frac{1}{2}}} \mathcal{J}_{\lambda-{\frac{1}{2}}}(zx),
	\end{equation}
where $x,z>0$ and $\mathcal{J}_{\nu}$ is the Bessel function of the first kind and order $\nu>-1$. Indeed, for $\lambda>-{\frac{1}{2}}$,
	\begin{equation}\label{B.Eigenvalues}
		\Delta_{\lambda}\varphi_{z}^{\lambda}=z^{2}\varphi_{z}^{\lambda},
	\end{equation}
for $z>0$.

\par The heat kernel associated to $\Delta_{\lambda}$ is
	\begin{equation}\label{B.Heat.Kernel.integral}
		W_{t}^{\lambda}(x,y)=\int\limits_{0}^{\infty} e^{-z^{2}t}\varphi_{z}^{\lambda}(x)\varphi_{z}^{\lambda}(y) d\mu_{\lambda}(z),
	\end{equation}
for $t,x,y>0$. Explicitly, the heat kernel is given by
	\begin{equation}\label{B.Heat.Kernel.explicit}
		W_{t}^{\lambda}(x,y)={\frac{(xy)^{-\lambda+{\frac{1}{2}}}}{2t}} e^{-\frac{(x^{2}+y^{2})}{4t}} \mathcal{I}_{\lambda-{\frac{1}{2}}}\left( {\frac{xy}{2t}}\right),
	\end{equation}
for $t,x,y>0$, where $\mathcal{I}_{\nu}$ is the modified Bessel function of the first kind and order $\nu>-1$. This function verifies (see \cite{L})
\begin{equation}\label{Modified.B.for.small.z}
		\mathcal{I}_{\nu}(z)\sim z^{\nu}, \qquad \qquad \mbox{ if } z<1,
\end{equation}
\begin{equation}\label{Modified.B.for.big.z}
		\mathcal{I}_{\nu}(z)\sim e^{z}z^{-\frac{1}{2}}, \qquad \qquad \mbox{ if } z>1.
\end{equation}

\par For a function $f$, its Bessel heat diffusion integral is
\begin{equation}\label{B.diffusion.integral}
W_{t}^{\lambda}f(x)=\int\limits_{0}^{\infty}W_{t}^{\lambda}(x,y)f(y)d\mu(y),
\end{equation}
for $t,x>0$; and its local maximal operator is defined for $a>0$ by
\begin{equation}\label{B.local.maximal.operator}
W^{\lambda,\ast}_{a}f(x)= \sup\limits_{0<t<a}|W_{t}^{\lambda}f(x)|,
\end{equation}
for $x>0$.

\par The Poisson kernel associated to $\Delta_{\lambda}$ is
	\begin{equation}\label{B.Poisson.Kernel.integral}
	P_{t}^{\lambda}(x,y)=\int\limits_{0}^{\infty} e^{-zt}\varphi_{z}^{\lambda}(x)\varphi_{z}^{\lambda}(y)d\mu_{\lambda}(z),
	\end{equation}
for $t,x,y>0$. By the subordination formula (see for example \cite{S}) we have that
		\begin{equation}\label{B.Poisson.Kernel.subordination.formula}
		P_{t}^{\lambda}(x,y)= \frac{t}{\sqrt{4\pi}} \int\limits_{0}^{\infty} e^{-\frac{t^{2}}{4u}} W_{t}^{\lambda}(x,y) \frac{du}{u^{\frac{3}{2}}}. %\int\limits_{0}^{\infty} W_{\frac{t^{2}}{4z}}^{\lambda}(x,y){\frac{e^{-z}}{\sqrt{\pi z}}}dz.
		\end{equation}
Explicitly, the Poisson kernel has the following expression (see Section 6 of \cite{BHNV}), in terms of ordinary hypergeometric $_{2}F_{1}$ functions:
		\begin{align}\label{B.Poisson.Kernel.explicit}
		P_{t}^{\lambda}(x,y)= & 2\pi^{-{\frac{1}{2}}}{\frac{\Gamma(\lambda+1)}{\Gamma(\lambda+{\frac{1}{2}})}} {\frac{t}{(x^{2}+y^{2}+t^{2})^{\lambda+1}}} \times \\ \nonumber & \qquad \times {}_{2}F_{1}\left( {\frac{\lambda+1}{2}} ; {\frac{\lambda+2}{2}} ; {\frac{2\lambda+1}{2}} ; \left( {\frac{2xy}{x^{2}+y^{2}+t^{2}}} \right)^{2} \right).
		\end{align}

\par For a function $f$, its Bessel Poisson integral is
\begin{equation}\label{B.diffusion.integral.P}
P_{t}^{\lambda}f(x)=\int\limits_{0}^{\infty}P_{t}^{\lambda}(x,y)f(y)d\mu(y),
\end{equation}
for $t,x>0$; and its local maximal operator is defined for $a>0$ by
\begin{equation}\label{B.local.maximal.operator.P}
P^{\lambda,\ast}_{a}f(x)= \sup\limits_{0<t<a}|P_{t}^{\lambda}f(x)|.
\end{equation}

\section{Conditions on data $f$ for almost everywhere convergence for the Bessel heat equation}\label{heat}

\par In this section we focus on the initial data problem \ref{B.heat.eqn}. In order to do computations we will use the following expressions for the Bessel heat kernel: from \ref{B.Heat.Kernel.explicit} and properties \ref{Modified.B.for.small.z} and \ref{Modified.B.for.big.z} we can write
		\begin{equation}\label{W.split.xy<2t}
		x^{2\lambda}W_{t}^{\lambda}(x,y)y^{2\lambda}\chi_{\left\{x,y>0:xy\le 2t\right\}}(x,y) \sim  {\frac{(xy)^{2\lambda}}{t^{\lambda+{\frac{1}{2}}}}} e^{-{\frac{(x^{2}+y^{2})}{4t}}} \chi_{\left\{x,y>0:xy\le 2t\right\}}(x,y),
	\end{equation}
	\begin{equation}\label{W.split.xy>2t}
		x^{2\lambda}W_{t}^{\lambda}(x,y)y^{2\lambda}\chi_{\left\{x,y>0:xy>2t\right\}}(x,y) \sim \frac{(xy)^{\lambda}} {(2t)^{\frac{1}{2}}} e^{-{\frac{(x-y)^{2}} {4t}}} \chi_{\left\{x,y>0:xy>2t\right\}}(x,y).
	\end{equation}
for $x>0$ and $y>0$.

\par The proof of Theorem \ref{thm:Main} in the heat context follows from the next three propositions. Let us begin by stating necessary and sufficient conditions on the initial data $f$  for the existence of $e^{-t\Delta_{\lambda}}$.

\begin{prop}\label{prop:B.Equivalent.Conditions}
Let $T>0$ fixed (possibly $T=\infty$) and $\lambda >-\frac{1}{2}$. For a measurable function $f:\RR^{+}\to \RR$, the following statements are equivalent
	\begin{itemize}
			\item[(i)] $\int\limits_{0}^{\infty} W_{t}^{\lambda}(x,y)|f(y)|d\mu_{\lambda}(y)<\infty$, for all $t\in(0,T)$ and $x>0$;
			\item[(ii)] $\int\limits_{0}^{\infty} W_{t}^{\lambda}(x_{t},y)|f(y)|d\mu_{\lambda}(y)<\infty$, for all $t\in (0,T)$ and some $x_{t}>0$;
			\item[(iii)] $\int\limits_{0}^{\infty} \phi_{t}^{\lambda}(y) |f(y)|dy<\infty$, for all $t\in(0,T)$, where
\begin{equation*}
\phi_{t}^{\lambda}(y)= y^{\lambda}\left(\frac{y}{y+1}\right)^{\lambda} e^{-{\frac{y^{2}}{4t}}}.
\end{equation*}
\end{itemize}
\end{prop}

\begin{proof}

\par That (i) implies (ii) is trivial. Let us prove that (ii) implies (iii). Pick $t_{0}<T$ and $x_{0}>0$ such that (ii) holds: $$\int\limits_{0}^{\infty} x_0^{2\lambda}W_{t_{0}}^{\lambda}(x_{0},y)f(y)y^{2\lambda}dy<\infty.$$
Then,	\begin{align*}
 \infty > &  \int\limits_{0}^{\infty} x_0^{2\lambda}W_{t_{0}}^{\lambda}(x_{0},y)f(y)y^{2\lambda}dy \sim
		\int\limits_{\left\{y>0:x_{0}y\le 2t_{0}\right\}}  {\frac{(x_0 y)^{2\lambda}}{t_{0}^{\lambda+{\frac{1}{2}}}}} e^{-{\frac{(x_{0}^{2}+y^{2})}{4t_{0}}}} f(y)dy  \\
		 & +\int\limits_{\left\{y>0:x_{0}y>2t_{0}\right\}}  {\frac{(x_{0}y)^{\lambda}}{(2t_{0})^{\frac{1}{2}}}} e^{-{\frac{(x_{0}-y)^{2}}{4t_{0}}}} f(y)dy \\
		 = &   e^{-{\frac{x_{0}^{2}}{4t_{0}}}}  {\frac{(x_0)^{2\lambda}}{t_{0}^{\lambda+{\frac{1}{2}}}}} \int\limits_{\left\{y>0:x_{0}y\le 2t_{0}\right\}}  y^{2\lambda} e^{-{\frac{y^{2}}{4t_{0}}}} f(y)dy  \\
		 & \qquad + {\frac{x_{0}^{\lambda}}{(2t_{0})^{\frac{1}{2}}}} \int\limits_{\left\{y>0:x_{0}y>2t_{0}\right\}} y^{\lambda}  e^{-{\frac{(x_{0}-y)^{2}}{4t_{0}}}} f(y)dy  = {I}_{1}+{I}_{2}.
	\end{align*}
For $\lambda\geq0$ we deduce that
\begin{align*}
\infty >& {I}_{1} \ge c(x_{0},t_{0},\lambda)\int\limits_{\left\{y>0:x_{0}y\le 2t_{0}\right\}}\left(\frac{1+y}{1+y}\right)^{\lambda} y^{2\lambda} e^{-{\frac{y^{2}}{4t_{0}}}} f(y)dy \\
\ge & c(x_{0},t_{0},\lambda) \int\limits_{\left\{y>0:x_{0}y\le 2t_{0}\right\}}\left(\frac{y}{1+y}\right)^{\lambda} y^{\lambda} e^{-{\frac{y^{2}}{4t_{0}}}} f(y)dy.
\end{align*}On the other hand, clearly
\begin{align*}
\infty > {I}_{2} \ge & c(x_{0},t_{0},\lambda) \int\limits_{\left\{y>0:x_{0}y>2t_{0}\right\}} y^{\lambda}  e^{-{\frac{(x_{0}-y)^{2}}{4t_{0}}}} f(y) dy \\
\geq & c(x_{0},t_{0},\lambda) \int\limits_{\left\{y>0:x_{0}y>2t_{0}\right\}} \left(\frac{y}{y+1}\right)^{\lambda}y^{\lambda}  e^{-{\frac{y^{2}}{4t_{0}}}} f(y) dy,\\
\end{align*}
since $e^{\frac{x_0y}{2t_0}}>e.$
For $ -\frac{1}{2}<\lambda<0$ we get
\begin{align*}
\infty > {I}_{1} \ge & c(x_{0},t_{0},\lambda) \int\limits_{\left\{y>0:x_{0}y\le 2t_{0}\right\}}\left(\frac{1}{y}\right)^{-\lambda} y^{\lambda} e^{-{\frac{y^{2}}{4t_{0}}}} f(y)dy\\
 = & c(x_{0},t_{0},\lambda) \int\limits_{\left\{y>0:x_{0}y\le 2t_{0}\right\}} \left(\frac{1+\frac{2t_0}{x_0}}{1+\frac{2t_0}{x_0}}\right)^{-\lambda}\left(\frac{1}{y}\right)^{-\lambda} y^{\lambda} e^{-{\frac{y^{2}}{4t_{0}}}} f(y)dy\\
\ge & c(x_{0},t_{0},\lambda) \int\limits_{\left\{y>0:x_{0}y\le 2t_{0}\right\}} \left(\frac{1+y}{1+\frac{2t_0}{x_0}}\right)^{-\lambda}\left(\frac{1}{y}\right)^{-\lambda} y^{\lambda} e^{-{\frac{y^{2}}{4t_{0}}}} f(y)dy \\
 \ge & c(x_{0},t_{0},\lambda) \int\limits_{\left\{y>0:x_{0}y\le 2t_{0}\right\}}\left(\frac{y}{1+y}\right)^{\lambda} y^{\lambda} e^{-{\frac{y^{2}}{4t_{0}}}} f(y)dy.
\end{align*}
Also
\begin{align*}
\infty > {I}_{2} \ge & c(x_{0},t_{0},\lambda) \int\limits_{\left\{y>0:x_{0}y>2t_{0}\right\}} y^{\lambda}  e^{-{\frac{y^{2}}{4t_{0}}}}f(y) dy \\
\ge & c(x_{0},t_{0},\lambda) \int\limits_{\left\{y>0:x_{0}y>2t_{0}\right\}} \left(\frac{x_0}{2t_0}+1\right)^{-\lambda} y^{\lambda}  e^{-{\frac{y^{2}}{4t_{0}}}} f(y) dy\\
\ge & c(x_{0},t_{0},\lambda) \int\limits_{\left\{y>0:x_{0}y>2t_{0}\right\}} \left(\frac{1}{y}+1\right)^{-\lambda} y^{\lambda}  e^{-{\frac{y^{2}}{4t_{0}}}} f(y) dy \\
 \ge & c(x_{0},t_{0},\lambda) \int\limits_{\left\{y>0:x_{0}y>2t_{0}\right\}}\left(\frac{y}{1+y}\right)^{\lambda} y^{\lambda}  e^{-{\frac{y^{2}}{4t_{0}}}} f(y) dy.
 \end{align*}
 Then we cover all cases and thus (iii) holds for all $t>0$.

\par Finally, let us show that (iii) implies (i). Pick $0<t<T$ and $x>0$. Splitting as before
\begin{align*}
&  \int\limits_{0}^{\infty} x^{2\lambda}W_{t}^{\lambda}(x,y)f(y)y^{2\lambda}dy  \sim \\
 & \qquad \qquad \sim   c(x,t,\lambda)\int\limits_{\left\{y>0:xy<2t\right\}} y^{2\lambda} e^{-{\frac{y^{2}}{4t}}}f(y)dy \quad   \\
 & \qquad \qquad \quad + c(x,t,\lambda)\int\limits_{\left\{y>0:xy>2t\right\}}  y^{\lambda} e^{-{\frac{(x-y)^{2}}{4t}}}   f(y)dy \quad = \quad {I}_{1}+{I}_{2}.
\end{align*}

\par Let us see first that $I_{1}$ is finite. Indeed, if $\lambda\geq 0$ it is inmediate that
\begin{align*}
{I}_{1} =&  c(x,t,\lambda) \int\limits_{\left\{y>0:xy\le 2t \right\}}\left(\frac{1+y}{1+y}\right)^{\lambda} y^{2\lambda} e^{-{{\frac{y^{2}}{4t}}}} f(y)dy  \\
\leq &   c(x,t,\lambda) \int\limits_{\left\{y>0:xy\le 2t\right\}}\left(\frac{1+\frac{2t}{x}}{1+y}\right)^{\lambda} y^{2\lambda} e^{-{\frac{y^{2}}{4t}}} f(y)dy\\
\leq & c(x,t,\lambda)\int\limits_{\left\{y>0:xy\le 2t\right\}}\left(\frac{y}{1+y}\right)^{\lambda} y^{\lambda} e^{-{\frac{y^{2}}{4t}}} f(y)dy.
\end{align*}
Also,  if $-\frac{1}{2}<\lambda<0$, we have
\begin{align*}
{I}_{1} \leq &  c(x,t,\lambda) \int\limits_{\left\{y>0:xy\le 2t\right\}}(1+y)^{-\lambda}\left(\frac{1}{y}\right)^{-\lambda} y^{\lambda} e^{\frac{-y^{2}}{4t}}  f(y)dy \\
= &   c(x,t,\lambda)\int\limits_{\left\{y>0:xy\le 2t\right\}}\left(\frac{y}{1+y}\right)^{\lambda} y^{\lambda} e^{\frac{-y^{2}}{4t}} f(y)dy.
\end{align*}

\par As for the finitude of $I_{2}$, let us recall estimate (3.4) from \cite{GHSTV}: for $x,y\in\RR$, $t>0$, $M>1$ we have that
	\begin{equation}
	\label{desigualdad.exp.GHSTV}
		{\frac{1}{c}}e^{-{\frac{|y|^{2}}{4t}}\left( {\frac{M+1}{M}} \right)^{2}}\le e^{-{\frac{|x-y|^{2}}{4t}}} \le ce^{-{\frac{|y|^{2}}{4t}}\left( {\frac{M-1}{M}} \right)^{2}},
	\end{equation}
where $c=c(x,t,M)$

\par Now if $\lambda\geq 0$, then
\begin{align*}
I_{2} = &  c(x,t,\lambda)\int\limits_{\left\{y>0:xy>2t\right\}}  y^{\lambda} e^{-{\frac{(x-y)^{2}}{4t}}}f(y)dy\\
\leq & c(x,t,\lambda, M)\int\limits_{\left\{y>0:xy>2t\right\}} \left(1+\frac{x}{2t}\right)^\lambda \left(1+\frac{x}{2t}\right)^{-\lambda} y^{\lambda}e^{-{{\frac{y^{2}}{4t}}}\left( {\frac{M-1}{M}} \right)^{2}}f(y)dy\\
\leq & c(x,t,\lambda, M)\int\limits_{\left\{y>0:xy>2t\right\}}\left(1+\frac{1}{y}\right)^{-\lambda} y^{\lambda}e^{-{{\frac{y^{2}}{4t}}}\left( {\frac{M-1}{M}} \right)^{2}}f(y)dy.
\end{align*}
Choosing $M$ big enough such that $t<t_{0}:=t\left(\frac{M}{M-1}\right)^2 < T$ and using (iii) we get that $I_2$ is finite. For $-\frac{1}{2}<\lambda<0$ we proceed  as above, but this time $I_2$ is bounded as follows
\begin{align*}
I_{2} \leq &c(x,t,\lambda, M)\int\limits_{\left\{y>0:xy>2t\right\}} \left(\frac{1+y}{1+y}\right)^{-\lambda} y^{\lambda}e^{-{\frac{|y|^{2}}{4t_{0}}}}f(y)dy\\
< &c(x,t,\lambda, M)\int\limits_{\left\{y>0:xy>2t\right\}} \left(1+\frac{1}{y}\right)^{-\lambda} y^{\lambda}e^{-{\frac{|y|^{2}}{4t_{0}}}}f(y)dy <\infty.
\end{align*}

\end{proof}

\begin{prop}\label{prop:B.solution.is.C.infinity}
Let  $f$ be a function satisfying the conditions in Proposition \ref{prop:B.Equivalent.Conditions}, and $\lambda\geq 0$, then
	\begin{equation}
	\label{B.solution.is.C.infinity}
		u(t,x)=\int\limits_{0}^{\infty}W_{t}^{\lambda}(x,y)f(y)d\mu_{\lambda}(y) \in C^{\infty}((0,T)\times\RR^{+}).
	\end{equation}
\end{prop}

\begin{proof}
\par We need to prove $$\int\limits_{0}^{\infty} |\partial_{t}^{\alpha}\partial_{x}^{\beta}W_{t}^{\lambda}(x,y)| f(y)d\mu_{\lambda}(y)<\infty$$ for all $\alpha,\beta\ge0$. Since the kernel $W_{t}^{\lambda}(x,y)$ satisfies the Bessel heat equation, we may only check the finitude of the integral for the derivatives on the $t$ variable:
\begin{equation}
\label{1}
\int\limits_{0}^{\infty} |\partial_{t}^{\alpha}W_{t}^{\lambda}(x,y)| f(y)d\mu_{\lambda}(y)<\infty.
\end{equation}
 Thus we need to compute ${\frac{\partial}{\partial t}}W_{t}^{\lambda}(x,y)$, and in order to do this let us recall that (see for example \cite{L})
\begin{equation}
\label{deriv.B}
{\frac{d}{dz}}\left(z^{-\nu}\mathcal{I}_{\nu}(z)\right)=z^{-\nu}\mathcal{I}_{\nu+1}(z).
\end{equation}
Keeping this formula in mind we rewrite the kernel in (\ref{B.Heat.Kernel.explicit}) as follows:
\begin{equation}
\label{B.Heat.Kernel.v2}
W_{t}^{\lambda}(x,y) =  {\frac{1}{(2t)^{\lambda+{\frac{1}{2}}}}} e^{-{\frac{x^{2}+y^{2}}{4t}}} \left[ \left( {\frac{xy}{2t}} \right)^{-(\lambda-{\frac{1}{2}})} \mathcal{I}_{\lambda-{\frac{1}{2}}} \left( {\frac{xy}{2t}} \right) \right].
\end{equation}
Then we have that
\begin{align*}
{\frac{\partial}{\partial t}}W_{t}^{\lambda}(x,y) = & \left[ -2{\frac{\lambda+{\frac{1}{2}}}{2t}} + \left( {\frac{x^{2}+y^{2}}{4t^{2}}} \right) \right] W_{t}^{\lambda}(x,y) + \left( -{\frac{(xy)^{2}}{2t^{2}}} \right) W_{t}^{\lambda+1}(x,y),
\end{align*}
from where
\begin{align*}
\int\limits_{0}^{\infty} {\frac{\partial}{\partial t}} W_{t}^{\lambda}(x,y)f(y)d\mu(y)= & \int\limits_{0}^{\infty} \left( -2{\frac{\lambda+{\frac{1}{2}}}{2t}} +  {\frac{x^{2}+y^{2}}{4t^{2}}} \right)  W_{t}^{\lambda}(x,y) f(y)d\mu(y)  \\
& + \int\limits_{0}^{\infty} -{\frac{(xy)^{2}}{2t^{2}}} W_{t}^{\lambda+1}(x,y)f(y)d\mu(y) ={I}_{1}+{I}_{2}.
\end{align*}

\par Let us first observe the integrand in $I_{2}$:
\begin{align*}
-\frac{(xy)^{2}}{2t^{2}} W_{t}^{\lambda +1}(x,y)y^{2\lambda} = -\frac{1}{2t^{2}x^{2\lambda}} W_{t}^{\lambda +1}(x,y)(xy)^{2(\lambda +1)}.
\end{align*}
Thus if we split ${I}_{2}$ as usual and apply estimates \ref{W.split.xy<2t} and \ref{W.split.xy>2t}, we obtain
\begin{align*}
{I}_{2} \sim &  \int\limits_{\left\{y>0:xy\le 2t \right\}} -{\frac{(xy)^{2}}{2t^{2}}} {\frac{1}{2t}} W_{t}^{\lambda}(x,y)  f(y)d\mu(y)  \\
 & + \int\limits_{\left\{y>0:xy>2t \right\}} -{\frac{xy}{2t^{2}}}  W_{t}^{\lambda}(x,y)  f(y)d\mu(y).
\end{align*}
Now, from this expression and ${I}_{1}$ we observe that (\ref{1}) follows if $y^{\alpha}f(y)$ for $\alpha\geq1$ satisfies the conditions of proposition \ref{prop:B.Equivalent.Conditions}, which clearly does.
\end{proof}

\begin{prop}\label{prop:B.solution.ae.limit}
If $f$ satisfies the conditions in Proposition \ref{prop:B.Equivalent.Conditions}, then
	\begin{equation}
	\label{B.solution.ae.limit}
		\lim\limits_{t\to 0^{+}} e^{-t\Delta_{\lambda}}f(x)=f(x), \qquad \mbox{ a.e. } x>0.
	\end{equation}
\end{prop}

\begin{proof}

\par Let us see first that for every $n_{0}\in\mathbb{NN}$, \ref{B.solution.ae.limit} holds for a.e. $0 <x\le n_{0}$. Let us split
$$f=f\chi_{|y|\le M} + f\chi_{|y|> M} =f_{1}+f_{2},$$ where $M$ is to be chosen later.

\par We have that, if $M>2n_{0}$,
\begin{align*}
e^{-t\Delta_{\lambda}}f_{2}(x) = & \int\limits_{|y|> M} W_{t}^{\lambda}(x,y)f(y)d\mu(y) \\
\sim & \int\limits_{|y|> M} \frac{y^{2\lambda}}{t^{\lambda+\frac{1}{2}}} e^{\frac{-(x^{2}+y^{2})}{4t}} \chi_{xy\le 2t}(y) f(y) dy   \\
& + \int\limits_{|y|> M} \frac{y^{\lambda}}{x^{\lambda}(2t)^{\frac{1}{2}}} e^{-{\frac{(x-y)^{2}}{4t}}} \chi_{xy> 2t}(y)  f(y)dy  \\
 \le &  c(\lambda) \int\limits_{|y|> M} \left(\frac{y^{2}}{4t}\right)^{\lambda} e^{-\frac{y^{2}}{4t}} e^{-\frac{x^{2}}{4t}} \frac{1}{t^{\frac{1}{2}}} \frac{y^{\lambda}}{(2n_{0})^{\lambda}} \chi_{xy\le 2t}(y) f(y) dy   \\
& + \int\limits_{|y|> M} \frac{y^{\lambda}}{x^{\lambda}(2t)^{\frac{1}{2}}}  \chi_{xy> 2t}(y) e^{\frac{-\left(\frac{y}{2}\right)^{2}}{4t}} f(y)dy, \\
\end{align*}
where we used the facts that $y>2n_{0}$ and $|x-y|>\frac{y}{2}$ for $y>2x$, respectively in each term of the sum. Since $\left(\frac{y^{2}}{ct} \right)^{\alpha} e^{-\frac{y^{2}}{4t}} \le c(\lambda)e^{-{\frac{\left(\frac{y}{2}\right)^{2}}{ct}}}$ for $c,\alpha>0$,
\begin{align*}
e^{-t\Delta_{\lambda}}f_{2}(x)  \le & c(n_{0},\lambda) \int\limits_{|y|>M} \frac{y^{\lambda}}{(2t)^{\frac{1}{2}}} e^{\frac{-y^{2}}{16t}} f(y) dy  \\
 \le & c(n_{0},\lambda) \int\limits_{|y|>M} y^{\lambda-1} e^{\frac{-\left(\frac{y}{2}\right)^{2}}{16t}} f(y) dy \le c(n_{0},\lambda) \int\limits_{|y|>M} y^{\lambda} e^{\frac{-y^{2}}{32t}} f(y) dy.
\end{align*}
Hence, choosing $t_0$ such that $t <t_{0}/8 < T$ and taking into account that  $\left( \frac{1+M}{M} \right)^{\lambda} \left( \frac{y}{y+1} \right)^{\lambda} \ge 1 $ for $|y|> M>1 $, we have
\begin{align*}
e^{-t\Delta_{\lambda}}f_{2}(x) \le & c(n_0,\lambda,M)  \int\limits_{|y|> M} \left( \frac{y}{y+1} \right)^{\lambda} y^{\lambda}e^{-\frac{y^{2}}{4t_{0}}}f(y)dy.
\end{align*}

Hence, from Proposition \ref{prop:B.Equivalent.Conditions}, given $\epsilon >0$ there exist $ M_0(\epsilon,n_0)>0 $ large enough  and $t_0>0 $ such that
\begin{align*}
e^{-t\Delta_{\lambda}}f_{2}(x) <\epsilon \qquad \mbox{for every\,$ M > M_0$, \,\, $t<t_0$ \,\,and \,\,$0<x \le n_0$}.
\end{align*}
\par On the other hand, since $e^{-t\Delta_{\lambda}}$ defines a symmetric diffusion semigroup as in \cite{S} and $f_{1}\in L^{1}(\RR^{+},d\mu)$, $$\lim\limits_{t\to 0^{+}}e^{-t\Delta_{\lambda}}f_{1}(x)=f_{1}(x)$$ for almost every $|x|\le n_{0}$, \ref{B.solution.ae.limit} holds.
\end{proof}

\par Thus Theorem \ref{thm:Main} follows. Next we give the proof for Corollary \ref{cor:B.weighted.class.characterization}.

\begin{proof}
\par If $v^{-\frac{1}{p}}\phi_{t}^{\lambda}\in L^{p'}(\RR^{+})$ and $f\in L^{p}(\RR^{+},v)$, then
		\begin{align*}
			\int\limits_{0}^{\infty} \phi_{t}^{\lambda}(y)f(y)dy = & \int\limits_{0}^{\infty} v^{-\frac{1}{p}}(y)\phi_{t}^{\lambda}(y)f(y)v^{\frac{1}{p}}(y)dy  \\
			\le & \left( \int\limits_{0}^{\infty} \left(v^{-\frac{1}{p}}(y)\phi_{t}^{\lambda}(y)\right)^{p'} dy \right)^{\frac{1}{p'}} \left( \int\limits_{0}^{\infty} f^{p}(y)v(y)dy \right)^{\frac{1}{p}} <\infty.
		\end{align*}
\par The converse follows by a duality argument, the Landau Theorem. See for example \cite{BS}.
\end{proof}

\par Now we give an answer to the last question, that is the boundedness of the local maximal operator on weighted $L^{p}(v)$, where $v$ is a weight in the class $D_{p}^{heat}(\Delta_{\lambda})$. First let us observe the integrability factor's behaviour:
 \begin{equation}\label{B.integrability.factor.behaviour}
\phi_{t}^{\lambda}(y)= y^{\lambda}\left(\frac{y}{y+1}\right)^{\lambda} e^{-{\frac{y^{2}}{4t}}}= \frac{y^{2\lambda}}{(y+1)^{\lambda}}e^{-{\frac{y^{2}}{4t}}}\sim\left\{
\begin{array}
[c]{l}%
y^{2\lambda}e^{-{\frac{y^{2}}{4t}}}, \,  \text{ if } y\leq 1;  \\
y^{\lambda}e^{-{\frac{y^{2}}{4t}}},\,
\text{ if }y > 1  .
\end{array}
\right.
\end{equation}
Next, we need an auxiliary estimate, which will use the following notation:  \begin{align*} a\vee b =& \max\{a,b\}, \\ a\wedge b = & \min\{a,b\}. \end{align*}

\begin{lem}\label{L.estimate.W}
For $\lambda>0$, $x,y>0$, $T>t>0$ and $M>1$, the following estimate holds:
\begin{align}\label{estimate.W}
y^{2\lambda}W_{t}^{\lambda}(x,y) \le & c(\lambda,M)\left((x\wedge 1)^{-2\lambda}W_{C_M t}^{\triangle}(x-y) \chi_{\{y\le M(x\vee1)\}}\frac{y^{2\lambda}}{(y+1)^{\lambda}} \right. \\
& \left. \quad \quad \quad \quad+ (x\wedge 1)^{-\lambda}\phi_{c_{M}t}^{\lambda}(y)\right),\nonumber
\end{align}
where $W_{t}^{\triangle}(x)=\frac{e^{-\frac{x^{2}}{4t}}}{(\pi t)^{\frac{1}{2}}}$ is the classical Laplace heat kernel and $C_{M}\downarrow 1$ as $M\downarrow 1$.
\end{lem}
\begin{proof}
\par From \ref{W.split.xy<2t} and \ref{W.split.xy>2t} it follows that
\begin{align*}
y^{2\lambda} W_{t}^{\lambda}(x,y) \sim & \frac{y^{2\lambda}}{t^{\lambda+\frac{1}{2}}} e^{-\frac{x^{2}+y^{2}}{4t}} \chi_{\{xy\le 2t\}}(y) + \frac{y^{\lambda}}{x^{\lambda}} \frac{e^{-\frac{(x-y)^{2}}{4t}}} {(2t)^{\frac{1}{2}}}\chi_{\{xy>2t\}}(y)=  A + B.
\end{align*}

\par Let us estimate $A$ first. Suppose that $y<1$. In this case we have that
\begin{align*}
 A  \leq & c(\lambda)  \frac{y^{2\lambda}}{x^{2\lambda}}\frac{x^{2\lambda}}{t^{\lambda}}e^{-\frac{x^{2}}{4Mt}}\frac{1}{(2t)^{\frac{1}{2}}}e^{-(1-\frac{1}{M})\frac{(x^2+y^{2})}{4t}} \leq c(\lambda,M)\frac{y^{2\lambda}}{x^{2\lambda}}\frac {e^{-\frac{(x-y)^{2}}{C_M 4t}}}{(2t)^{\frac{1}{2}}}\\
 \leq &  c(\lambda,M)\frac{y^{2\lambda}}{(x\wedge1)^{2\lambda}}W_{C_M t}^{\triangle}(x-y).
\end{align*}
On the other hand, if $y> 1$, since $xy \leq 2t$,  we obtain that
\begin{equation}\label{A1}
 A  \leq  c(\lambda)\frac{y^{2\lambda}}{(xy)^{\lambda}}\frac{1}{(2t)^{\frac{1}{2}}}e^{-\frac{(x^2+y^{2})}{4t}}\leq  c(\lambda)\left(\frac{y\vee1}{x\wedge1}\right)^{\lambda}W_{t}^{\triangle}(x-y).
\end{equation}

\par To estimate $B$, assume first that $y> 1$, thus
\begin{equation}\label{B1}
B\leq c(\lambda)\left(\frac{y\vee1}{x\wedge 1}\right)^{\lambda}W_{t}^{\triangle}(x-y).
\end{equation}
If $y\leq 1$ and $2 y\leq x$, then $x-y\geq\frac{x}{2}$. Therefore, we have
\begin{align*}
B \leq &  c(\lambda)\left(\frac{y}{x}\right)^{\lambda}\frac {1}{(2t)^{\frac{1}{2}}}e^{-\frac{1}{M}\frac{(x-y)^{2}}{4t}}e^{-(1-\frac{1}{M})\frac{(x-y)^{2}}{4t}}\\
\leq & c(\lambda) \left(\frac{y}{x}\right)^{\lambda}\left(\frac{t}{x^2}\right)^{\lambda} \left(\frac{x^2}{t}\right)^{\lambda} e^{-\frac{1}{M}\frac{x^{2}}{16t}}\frac {1}{(2t)^{\frac{1}{2}}}e^{-\frac{(x-y)^{2}}{4C_Mt}}\\
\leq &  c(\lambda,M) \left(\frac{y}{x}\right)^{\lambda}\left(\frac{t}{x^2}\right)^{\lambda} W_{C_M t}^{\triangle}(x-y) \\
\leq & c(\lambda,M) \left(\frac{y\wedge1}{x\wedge1}\right)^{2\lambda}W_{C_M t}^{\triangle}(x-y).
\end{align*}
where in the last inequality we use that $xy>2t$. Now, If $y\leq 1$ and $2 y> x$, we clearly have that
\begin{align*}
B\leq  c(\lambda)\left(\frac{y\wedge1}{x \wedge 1}\right)^{2\lambda}W_{t}^{\triangle}(x-y).
\end{align*}
Hence, collecting the above inequalities,  we get that
\begin{align*}
y^{2\lambda}W_{t}^{\lambda}(x,y) \le c(\lambda,M) \frac{1}{({x \wedge 1})^{2\lambda}}W_{C_Mt}^{\triangle}(x-y)\frac{y^{2\lambda}}{(y+1)^{\lambda}}.
\end{align*}
\par Finally, note that if  $y> M(x\vee1)> 1$, then $y-x > y(\frac{M-1}{M})$. Hence, from \ref{A1} and \ref{B1}, we obtain that
\begin {align*}
y^{2\lambda}W_{t}^{\lambda}(x,y) \leq& c(\lambda)(x\wedge 1)^{-\lambda}\quad W_{t}^{\triangle}(x-y)\frac{y^{2\lambda}}{(y+1)^{\lambda}}\\
\leq &  c(\lambda)(x\wedge 1)^{-\lambda}\frac{1}{(2t)^{\frac{1}{2}}}e^{-\frac{(M-1)^2y^{2}}{4tM^2}}\frac{y^{2\lambda}}{(y+1)^{\lambda}}\\
\leq & c(\lambda)(x\wedge 1)^{-\lambda}(2t)^{-\frac{1}{2}}e^{-\frac{1}{4t M}} e^{-(\frac{M-1}{M})^3\frac{y^{2}}{4t}}\frac{y^{2\lambda}}{(y+1)^{\lambda}}\\
\leq & c(\lambda,M)(x\wedge 1)^{-\lambda}\phi_{c_{M}t}^{\lambda}(y),
\end{align*}
which finishes the proof.
\end{proof}

%\par The next theorem can be proved using the vector-valued theory of Rubio de Francia (see Proof of Theorem 2.1 of \cite{GHSTV}).

\begin{thm}\label{thm:Lpv.Lpu.boundness.local.maximal}
Let $1<p<\infty$ and $R>1$. Let $v$ be a weight such that $v^{-\frac{1}{p}}\in L^{p'}_{loc}(\RR)$. Then there exists a weight $u$ such that the local maximal operator defined as
\begin{equation}\label{local.maximal.operator}
\mathcal{M}_{R}^{loc}f(x)=\sup\limits_{x\in B_{r}} \int\limits_{B_{r}(x)}f(y)\chi_{|y|<Rx}(y)dy
\end{equation}
is bounded from $L^p(v)$ to $L^p(u)$. Moreover,
\begin{itemize}
			\item[(i)] If $||v^{-\frac{1}{p}}e^{-A|y|^{2}}||_{p'}<\infty$ for all $A>A_{0}$, where $A_{0}\ge 0$ is fixed, then for every $\sigma<1$ we can find a weight $u$ such that also \begin{align}\label{u1}
||u^{-\frac{\sigma}{p}}e^{-Ay^{2}}||_{p'}<\infty \end{align} for all $A>A_{0}\sigma R^{2}$. In particular, if $A_{0}=0$ or $\sigma<\frac{1}{R^{2}}$ then \ref{u1} holds for all $A>A_{0}$.

\item[(ii)] If we define $\|v\|_{D_{p}}=||v^{-\frac{1}{p}}(y\vee1)^{-2}||_{p'}<\infty$, then for every $\sigma<1$ we can find a weight $u$ such that also
    \begin{align}\label{u2}
||u^{-\frac{\sigma}{p}}(y\vee1)^{-2}||_{p'}<\infty. \end{align}
\end{itemize}
\end{thm}
\begin{proof}
The proof of item (i) is contained in Theorem 2.1 of \cite{GHSTV}. We only check (\ref{u2}). In order to do this,
we  recall the estimates and the constants defined in that paper. By the factorization Theorem of Rubio de Francia (see \cite[Thm. VI.4.2]{GR}),
we can assure the existence of some weight $U_k$, supported in a interval
$E_k$, such that $\|U_k^{-1}\|_{L^{\frac s{p-s}}}\leq 1$, $s< 1$ and
\[\int_{E_k} \big|\mathcal{M}_{R}^{loc}f(x)\big|^p\,U_k(x)\,dx\, \leq \,
C^p_k\,\|f\|^p_{L^p(v)}  ,
\]
where
\[ E_0=\{|x|< 1\} ,\quad E_k=\{2^{k-1}\leq |x|<2^k\}, \quad
k=1,2,\ldots,\]
and
$C_k=c_{s,p}|E_k|^{\frac1s-1}V_k$, with the constants $V_k$ given by
\begin{align*}
V_k = & \left|\left|v^{-\frac1p}\,\,\chi_{\{|y|<Rb^k\}}\right|\right|_{p'}=
\left|\left|v^{-\frac1p}\,|(y\vee1)|^{-2}|(y\vee1)|^2\,\chi_{\{|y|<R2^k\}}\right|\right|_{p'}\\
\leq & \|v\|_{D_{p}}\,(R2^{k})^2.
\end{align*}
In this way, to obtain the bound of $\mathcal{M}_{R}^{loc}$ it suffices to consider the
weight $u$ defined by
\begin{equation}\label{ucons}
u(x)=\sum_{k=0}^\infty \frac1{(2^{\gamma
k} C_k)^p}\,U_k(x)\chi_{E_k}(x),
 \end{equation}
 for some $\gamma>0$ to be determined later.

Given $\sigma<1$, we first select $s<1$ such that
$\frac{\sigma p'}p=\frac{s}{p-s}$. Then,
\begin{align}\label{auxs}
\|u^\sigma\|_{D_{p}}^{p'}  = & \int_{\RR} u(y)^{-\frac {\sigma
p'}p}\,|y|^{-2p'}\,dy\, = \, \sum_{k=0}^\infty\big(2^{\gamma k}
C_k\big)^{\sigma p'}\,\int_{E_k} U_k(y)^{-\frac
s{p-s}}\,|y|^{-2p'}\,dy \\\nonumber
 \leq & c\,\sum_{k=0}^\infty \Big(2^{\gamma
k}|E_k|^{\frac1s-1}\,2^{2k}\Big)^{\sigma
p'}\, 2^{-2kp'}= \sum_{k=0}^\infty 2^{-kp'\left(2(1-\sigma)-\gamma-\frac{1-\sigma}{p'}\right)},
\end{align}
where in the last inequality we have used the facts that $\|U_k^{-1}\|_{L^{\frac s{p-s}}}\leq 1$ and $(\frac{1}{s}-1)\sigma p'= 1-\sigma$.

\par Now, this series is convergent provided that $0<\gamma < (1-\sigma)\left(1+\frac{1}{p}\right)$.

 \end{proof}

\par To end this section we give the proof of Theorem \ref{thm:Lpv.Lpu.boundness.maximal.B.heat.operator}, item (i).

\begin{proof}
\par Let us fix $1<p<\infty$ and $a>0$. Let $v\in D_{p}^{heat}(\Delta_{\lambda})$. This means that $||v^{-\frac{1}{p}}\phi_{t}^{\lambda}||_{p'}<\infty$, where $\phi_{t}^{\lambda}(y)$ is the integrability factor given in proposition \ref{prop:B.Equivalent.Conditions}, namely
\begin{equation}\label{Ei}
\phi_{t}^{\lambda}(y)= y^{\lambda}\left(\frac{y}{y+1}\right)^{\lambda} e^{-{\frac{y^{2}}{4t}}}\sim (y\wedge 1)^{2\lambda} (y\vee 1)^{\lambda}   e^{-{\frac{y^{2}}{4t}}}.
\end{equation}
\par We need to show that the local maximal operator $W^{\lambda,\ast}_{a}$ defined in \ref{B.local.maximal.operator} maps $L^{p}(v)$ to $L^{p}(u)$ boundedly, for some weight $u$. Moreover, if $\sigma<1$ we need to find a weight $u$ such that for all $t>0$,
\begin{equation}
||u^{-\frac{\sigma}{p}}\phi_{t}^{\lambda}||_{p'}<\infty.
\end{equation}

\par For $f\in L^{p}(v)$ and $M>1$ to be chosen later, let us split as follows:
\begin{align*}
W_{a}^{\lambda,\ast}f(x) \le & \sup\limits_{0<t<a} \int W_{t}^{\lambda}(x,y)f(y)\chi_{\{y\le Mx \}}(y)d\mu(y)  \\
& + \sup\limits_{0<t<a} \int W_{t}^{\lambda}(x,y)f(y)\chi_{\{y>Mx\}}(y)d\mu(y)  \\
= & Af(x)+Bf(x).
\end{align*}
\par For $Af(x)$ let us recall Lemma \ref{L.estimate.W} and write
\begin{align*}
Af(x) \le & (c\lambda,M)(x\vee 1)^{-2 \lambda} \\ & \quad \sup\limits_{0<t<a} \int W_{C_M t}^{\triangle}(x-y)\frac{y^{2\lambda}}{(y+1)^{\lambda}}f(y)\chi_{\{ y\le Mx\}}(y)d(y).
\end{align*}

\par By a standard argument of slicing into dyadic shells, it follows that $Af(x)\le c(\lambda,M)(x\vee 1)^{-2\lambda} \mathcal{M}_{M}^{loc}\left(\frac{y^{2\lambda}}{(y+1)^{\lambda}}f\right)(x)$, where ${M}_{M}^{loc}f$ is the local maximal function considered in \ref{local.maximal.operator}, extended to $|y|\le M(|x|\vee 1)$. From Theorem \ref{thm:Lpv.Lpu.boundness.local.maximal} item (i), if we set $\sigma_{0}=\frac{1}{M^{2}}<1$ in the case $T<\infty$,  $\widetilde{f}=\frac{y^{2\lambda}}{(y+1)^{\lambda}}f$ and
$\tv=\left(\frac{y^{2\lambda}}{(y+1)^{\lambda}}\right)^{-p}v$, then for all $\sigma\le\sigma_{0}<1$ there exists a weight $\tu$ such that $$||\mathcal{M}_{M}^{loc}(\widetilde{f})||_{L^{p}(\tu)}\le c ||\widetilde{f}||_{L^{p}(\widetilde{v})}= c ||f||_{L^{p}(v)}$$ and
\begin{equation}\label{u}
\left|\left|(\tu)^{-\frac{\sigma}{p}}e^{-{\frac{y^{2}}{4t}}}\right|\right|_{p'}<\infty.
\end{equation}
provided that  $\left|\left|(\tv)^{-\frac{1}{p}}e^{-{\frac{y^{2}}{4t}}}\right|\right|_{p'}<\infty$ for all $t<T$ . This is true because  $v\in D_{p}^{heat}(\Delta_{\lambda})$.

\par Now choosing $u_1 (x) =(x\wedge 1)^{2\lambda p}\tu(x)$, we have
\begin{align*}
\left|\left|u_1^{-\frac{\sigma}{p}}\phi_{t}^{\lambda}\right|\right|_{p'}=&\left|\left|(\tu)^{-\frac{\sigma}{p}}(x\vee 1)^{-\sigma 2\lambda }\phi_{t}^{\lambda}\right|\right|_{p'}\\
\le & \left|\left|(\tu)^{-\frac{\sigma}{p}}x ^{2(1-\sigma)\lambda} \chi_{\{x\le 1\}}\right|\right|_{p'} + \left|\left|(\tu)^{-\frac{\sigma}{p}}x^{\lambda}  e^{-{\frac{x^{2}}{4t}}}\chi_{\{x\ge 1\}}\right|\right|_{p'} \\
\leq & \left|\left|(\tu)^{-\frac{\sigma}{p}}  e^{-{\frac{x^{2}}{4t(1+\epsilon)}}}\chi_{\{x\geq1\}}\right|\right|_{p'}.\\
\end{align*}
which is finite in view of  (\ref{u}), by choosing $\epsilon$ small enought such that $a(1+\epsilon)< T$.
This proves that $u_1^\sigma\in D_{p}^{heat}(\Delta_{\lambda})$. Clearly, when $T=\infty$, by applying Theorem \ref{thm:Lpv.Lpu.boundness.local.maximal} item (i) with $A_0 =0$ we can choose any $\sigma<1$ to obtain the same conclusion.

\par We now estimate $Bf$. From Lemma \ref{L.estimate.W}, we get
 \begin{align*}
Bf(x)\leq & \frac{c(\lambda,M)} {(x\wedge1)^{\lambda}}\sup\limits_{0<t<a} \int \phi_{C_{M}t}^{\lambda}(y)\chi_{\{y>Mx\}}(y).
 \end{align*}
 Choosing $M>1$ such that $t_{0}=C_M a<T$ when $T<\infty$, and applying H\"{o}lder's inequality, we have
\begin{align*}
 \int \phi_{t_0}^{\lambda}(y)f(y)dy = & \int \phi_{t_0}^{\lambda}(y)v^{-\frac{1}{p}}f(y)v^{\frac{1}{p}}dy  \\
 \le &  \left(\int \left(\phi_{t_0}^{\lambda}(y)v^{-\frac{1}{p}}\right)^{p'} dy \right)^{\frac{1}{p'}} \left( \int f^{p}(y)v(y)dy \right)^{\frac{1}{p}}  \\
 = & \left|\left|v^{-\frac{1}{p}}\phi_{t_0}^{\lambda}\right|\right|_{p'} ||f||_{L^{p}(v)}.
\end{align*}
We note that if $T=\infty$, then any $M>1$ works.
\par It follows that $$Bf(x)\le \frac{c(\lambda,M)}{(x\wedge1)^{\lambda}} ||f||_{L^{p}(v)}= c(x)||f||_{L^{p}(v)}.$$
Thus, setting a weight $u_{2}(x)\le \frac{1}{c(x)^p (1+x)^{p}}$, we see that
$$||Bf||_{L^{p}(u_{2})} \le c ||f||_{L^{p}(v)};$$ and recalling the behaviour of the integrating function $\phi_{t}^{\lambda}$ given in \ref{Ei} we also see that
$$\left|\left|u_2^{-\frac\sigma p}\, \phi_{t}^{\lambda} \right|\right|_{p'}\leq \left|\left|x^{\lambda(2-\sigma)}\chi_{\{x\le 1\}}\right|\right|_{p'} +\left|\left|(1+x)^{\sigma}x^{\lambda}e^{-\frac{|x|^2}{4t}}\right|\right|_{p'}<\infty$$
for all $t<T$.
\par The Theorem follows by taking $u(x)=\min\{u_{1}(x),u_{2}(x)\}$.

\end{proof}

\section{Conditions on data $f$ for almost everywhere convergence for the Bessel Poisson equation}\label{Poisson}

\par In this section we focus our attention on the initial value problem for the Poisson equation. Let us start with computing two estimates for the Poisson kernel which will be useful.

\par Throughout this section, we will consider $\lambda\geq0$. Recall also that we consider the integrating factor given by
\begin{equation}\label{P.integrability.factor}
\phi^{\lambda}(y)=\frac{y^{2\lambda}}{(y^{2}+1)^{\lambda+1}}\sim\left\{
\begin{array}
[c]{l}%
y^{2\lambda}, \,  \text{ if } y\leq 1  \\
y^{-2},\,
\text{ if }y > 1  .
\end{array}
\right.
\end{equation}

\par The first estimate is the same as shown in Proposition 4 of \cite{BHNV}, where the proof is based on the expression \ref{B.Poisson.Kernel.explicit} of the kernel and  behaviour properties of the hypergeometric functions.

\begin{lem}\label{lem:B.Kernel.Estimate.P}
Given $t>0$ and $x>0$ there exists a constant $c_{\lambda}$ such that
\begin{equation}\label{B.Kernel.Estimate.P}
\frac{c_{\lambda}^{-1}t}{[(x-y)^{2}+t^{2}](x^{2}+y^{2}+t^{2})^{\lambda}} \le P_{t}^{\lambda}(x,y) \le  \frac{c_{\lambda}t}{[(x-y)^{2}+t^{2}](x^{2}+y^{2}+t^{2})^{\lambda}}.
\end{equation}
\end{lem}
\begin{proof}
\par Recall that, from the explicit expression for the Bessel heat kernel \ref{B.Heat.Kernel.explicit},
\begin{align*}
P_{t}^{\lambda}(x,y) = & \frac{t}{\sqrt{4\pi}} \int\limits_{0}^{\infty} e^{-\frac{t^{2}}{4u}} W_{u}^{\lambda}(x,y) \frac{du}{u^{\frac{3}{2}}}  \\
= & \frac{t}{4\sqrt{\pi}} (xy)^{-\lambda+\frac{1}{2}} \int\limits_{0}^{\infty} e^{-\frac{(x^{2}+y^{2}+t^{2})}{4u}} I_{\lambda-\frac{1}{2}}\left(\frac{xy}{2u}\right) \frac{du}{u^{\frac{5}{2}}},
\end{align*}
and if we perform the variable change given by $z=\frac{1}{u}$, we have that
\begin{align}\label{Poisson.expression}
P_{t}^{\lambda}(x,y)=& \frac{t}{4\sqrt{\pi}} (xy)^{-\lambda+\frac{1}{2}} \int\limits_{0}^{\infty} e^{-\frac{(x^{2}+y^{2}+t^{2})}{4}z} I_{\lambda-\frac{1}{2}}\left(\frac{xy}{2}z\right) z^{\frac{1}{2}} dz.
\end{align}
Thus, if we change variables by $\frac{xy}{2}z\longrightarrow z$,
\begin{align} \label{Poisson.expression.alphabeta}
P_{t}^{\lambda}(x,y) =& c {\frac{t}{\alpha^{\lambda+1}}} \int\limits_{0}^{\infty} e^{-\frac{\beta}{\alpha}z} I_{\lambda-1}(z) z^{\frac{1}{2}} dz,
\end{align}
where $\beta=\frac{x^{2}+y^{2}+t^{2}}{4}$ and $\alpha=\frac{xy}{2}$. Now, if we split the integral in $0<z<1$ and $1<z$ and apply properties \ref{Modified.B.for.small.z} and \ref{Modified.B.for.big.z} of the modified Bessel functions,  we get that
\begin{align*}
P_{t}^{\lambda}(x,y) \sim & c \frac{t}{\alpha^{\lambda+1}} \int\limits_{0}^{1} e^{-\frac{\beta}{\alpha}z} z^{\lambda} dz +  c \frac{t}{\alpha^{\lambda+1}} \int\limits_{1}^{\infty} e^{-\frac{(\beta-\alpha)}{\alpha}z}  dz =I_{1}+I_{2}.
\end{align*}
Since $\frac{\beta}{\alpha}>1$, by changing variables according to $v=\frac{\beta}{\alpha}z$, we can write
\begin{align*}
I_{1}=c \frac{t}{\beta^{\lambda+1}}\left(\int\limits_{0}^{1} e^{-v}v^{\lambda}dv+\int\limits_{1}^{\frac{\beta}{\alpha}}e^{-v}v^{\lambda}dv \right).
\end{align*}
Since if $1<v<\frac{\beta}{\alpha}$ then $e^{-v}\le e^{-v}v^{\lambda}<c_{\lambda} e^{-\frac{v}{2}}$, it follows that $\int\limits_{1}^{\frac{\beta}{\alpha}}e^{-v}v^{\lambda}dv\sim c_{\lambda}$. Thus, $$I_{1}\sim c_{\lambda} \frac{t}{\beta^{\lambda+1}}=c_{\lambda} \frac{t}{(x^{2}+y^{2}+t^{2})^{\lambda+1}}.$$ On the other hand,
%\sim c \frac{t}{\alpha^{\lambda+1}} \int\limits_{0}^{\infty} e^{-\frac{\beta-\alpha}{\alpha}z} dz,
\begin{align*}
I_{2} = & c  t \frac{e^{-\frac{\beta-\alpha}{\alpha}}}{\alpha^{\lambda}(\beta-\alpha)} = c \left(\frac{\beta}{\alpha}\right)^{\lambda} e^{-\frac{\beta}{\alpha}} \frac{t}{\beta^{\lambda}(\beta-\alpha)}.
\end{align*}
It is clear that
$$
I_{1}+I_{2}\leq\frac{c_{\lambda}t}{(x^{2}+y^{2}+t^{2})^{\lambda}[(x-y)^{2}+t^{2}]}.
$$
Now let us see the estimate from below.
\begin{align*}
I_{1}+I_{2} = &\frac{c_{\lambda}t}{\beta^\lambda}\left(\left(\frac{\beta}{\alpha}\right)^{\lambda} e^{-\frac{\beta}{\alpha}} \frac{1}{\beta-\alpha}+\frac{1}{\beta}\right)
= \frac{c_{\lambda}t}{\beta^\lambda}\frac{\left(\frac{\beta}{\alpha}\right)^{\lambda} e^{-\frac{\beta}{\alpha}}+ 1 - \frac{\alpha}{\beta}}{\beta- \alpha}
\end{align*}
Taking $z=\frac{\beta}{\alpha}> 1$, we note that
$$
A = z^{\lambda} e^{-z}+1 - z^{-1} \longmapsto 1\quad\mbox{when $z\longrightarrow \infty$}\quad.
$$
Thus, there exists $N>1$ such that $A >1/2$ for every $z > N$ and clearly $A > (e^{-N}\wedge\frac{1}{2})$.
Hence the estimate holds.

\end{proof}

\par From this Lemma \ref{lem:B.Kernel.Estimate.P} the second estimate follows:

\begin{lem}\label{L.estimate.P}
For $\lambda>0$, $x,y,t>0$ and $M>1$ the following estimate holds
\begin{align*}
y^{2\lambda}P_{t}^{\lambda}(x,y)\le & c_{\lambda,M}  \left((x\wedge1)^{-2\lambda} P_{t}^{\triangle}(x-y)(y\wedge1)^{2\lambda}\chi_{y\le M(x\vee 1)}(y) + t \phi_{\lambda}(y) \right),
\end{align*}
where $P_{t}^{\Delta}(x)=c\frac{t}{t^{2}+x^{2}}$ is the classical Poisson kernel.
\end{lem}
\begin{proof}
\par From Lemma \ref{lem:B.Kernel.Estimate.P}, we have that
\begin{align*}
y^{2\lambda}P_{t}^{\lambda}(x,y) \chi_{y\le M (x\vee 1)}(y) \le & c_{\lambda} P_{t}^{\Delta}(x-y) \frac{y^{2\lambda}}{(x^{2}+y^{2}+t^{2})^{\lambda}} \chi_{y\le M (x\vee 1)}(y)\\
\le &  c_{\lambda}(x\wedge1)^{-2\lambda} P_{t}^{\Delta}(x-y) \chi_{y\le M (x\vee 1)}(y\wedge1)^{2\lambda}.
\end{align*}

\par Now, again from Lemma \ref{lem:B.Kernel.Estimate.P}, we have that
\begin{align*}
y^{2\lambda}P_{t}^{\lambda}(x,y) \chi_{y> M (x\vee 1)}(y) \le &  \frac{c_{\lambda}t y^{2\lambda}}{[(x-y)^{2}+t^{2}](x^{2}+y^{2}+t^{2})^{\lambda}} \chi_{y> M (x\vee 1)}(y).
\end{align*}
\par Since $y>M(x\wedge 1)$, hence $y>1$, it follows that $\frac{1}{(x-y)^{2}+t^{2}}\le \frac{c_{M}}{y^{2}}$ and also $\frac{1}{x^{2}+y^{2}+t^{2}}\le \frac{c_{M}}{y^{2}}$. Thus
\begin{align*}
y^{2\lambda}P_{t}^{\lambda}(x,y) \chi_{y> M (x\vee 1)}(y) \le &  c_{\lambda,M} \frac{1}{y^{2}}\le  c_{\lambda,M} \phi_{\lambda}(y).
\end{align*}
\par Thus we get the desired estimate.
\end{proof}

\par Next we follow the same steps as those for the heat problem: in order to prove Theorem \ref{thm:Main} in the Poisson context, we need three propositions. First, we need to characterize the Bessel Poisson integral by an integrability factor.

\begin{prop}\label{prop:B.Equivalent.Conditions.P}
The following statements are equivalent
	\begin{itemize}
			\item[(i)] $\int\limits_{0}^{\infty} P_{t}^{\lambda}(x,y)|f(y)|d\mu_{\lambda}(y)<\infty$, for all $t>0$ and $x>0$.
			\item[(ii)] $\int\limits_{0}^{\infty} P_{t}^{\lambda}(x_{t},y)|f(y)|d\mu_{\lambda}(y)<\infty$, for all $t>0$ and some $x_{t}>0$.
			\item[(iii)] $\int\limits_{0}^{\infty} \phi^{\lambda}(y) |f(y)|dy<\infty$,  where
\begin{equation}\label{B.integrability.factor.P}
\phi^{\lambda}(y)= {\frac{y^{2\lambda}}{(y^{2}+1)^{\lambda+1}}}.
\end{equation}
\end{itemize}

\end{prop}

\begin{proof}

\par Observe that (i) trivially implies (ii).

\par Let us prove that (ii) implies (iii). Fix $t>0$. We have from estimate \ref{B.Kernel.Estimate.P} that for some $x=x_{t}>0$, $$\int\limits_{0}^{\infty} P_{t}^{\lambda}(x_{t},y)|f(y)|d\mu_{\lambda}(y) \sim c_{\lambda} \int\limits_{0}^{\infty} \frac{ty^{2\lambda}}{[(x-y)^{2}+t^{2}](x^{2}+y^{2}+t^{2})^{\lambda}} dy <\infty.$$

\par If $x^{2}+t^{2}\le 1$, then $y^{2}+1\ge x^{2}+y^{2}+t^{2}$. If $x^{2}+t^{2}> 1$, then $y^{2}+1 \ge \left(\frac{1}{x^{2}+t^{2}}\right)y^{2}+1=\frac{1}{x^{2}+t^{2}}(x^{2}+y^{2}+t^{2})$. Also, $x^{2}+y^{2}+t^{2}\ge (x-y)^{2}+t^{2}$, thus $$\frac{(x^{2}+t^{2})\vee 1}{y^{2}+1}\le \frac{1}{x^{2}+y^{2}+t^{2}}\le\frac{1}{(x-y)^{2}+t^{2}}.$$ Hence,
\begin{align*}
\infty > & \int\limits_{0}^{\infty} P_{t}^{\lambda}(x_{t},y)|f(y)|d\mu_{\lambda}(y) \ge c_{\lambda} ((x^{2}+t^{2})\vee 1) t \int\limits_{0}^{\infty} \frac{y^{2\lambda}}{(y^{2}+1)^{\lambda+1}} f(y) dy  \\
& = c_{\lambda,x,t} \int\limits_{0}^{\infty} \phi^{\lambda}(y) f(y) dy.
\end{align*}

\par We will now show that (iii) implies (i). Take $t>0$ and $x>0$. We use estimate \ref{B.Kernel.Estimate.P} again.
\begin{align*}
& \int\limits_{0}^{\infty} P_{t}^{\lambda}(x,y) f(y) d\mu_{\lambda}(y)  \sim c_{\lambda} \int\limits_{0}^{\infty} \frac{ty^{2\lambda}}{[(x-y)^{2}+t^{2}](x^{2}+y^{2}+t^{2})^{\lambda}} dy  \\
& = c_{\lambda} \int\limits_{y\le M(x\vee1)} \frac{ty^{2\lambda}}{[(x-y)^{2}+t^{2}](x^{2}+y^{2}+t^{2})^{\lambda}} dy  \\
& \qquad + c_{\lambda} \int\limits_{y> M(x\vee1)} \frac{ty^{2\lambda}}{[(x-y)^{2}+t^{2}](x^{2}+y^{2}+t^{2})^{\lambda}} dy  \\
& \qquad + c_{\lambda} \int\limits_{y> M(x\vee1)} \frac{ty^{2\lambda}}{[(x-y)^{2}+t^{2}](x^{2}+y^{2}+t^{2})^{\lambda}} dy  \\
& =I_{1}+I_{2},
\end{align*}
for any $M>1$.

\par Let us consider first $I_{2}$. We have that $y>Mx$, and it is easy to see that $\frac{M-1}{M} y \le |y-x|$. Thus, $$\left(\frac{M-1}{M}\right)^{2}y^{2}+t^{2} \le (x-y)^{2}+t^{2}.$$

\par If $\left(\frac{M-1}{M}\right)^{2} \le t^{2}$ then $y^{2}+1\le \left(\frac{M}{M-1}\right)^{2} [(x-y)^{2}+t^{2}]$, and if $\left(\frac{M-1}{M}\right)^{2} \ge t^{2}$ then $y^{2}+1\le \frac{1}{t^{2}} [(x-y)^{2}+t^{2}]$. Hence $y^{2}+1\le \frac{1}{\left(\frac{M-1}{M}\right)^{2}\wedge t^{2}} [(x-y)^{2}+t^{2}]$. Again, since $x^{2}+y^{2}+t^{2}\ge (x-y)^{2}+t^{2}$, we have that $$ \frac{1}{(x^{2}+y^{2}+t^{2})^{\lambda}[(x-y)^{2}+t^{2}]} \le c_{x,t,M} \frac{1}{(y^{2}+1)^{\lambda+1}}.$$

\par As for $I_{1}$, it is finite because the function $y\to \frac{y^{2\lambda}}{[(x-y)^{2}+t^{2}](x^{2}+y^{2}+t^{2})^{\lambda}}$ is continuous on the compact region $y\le M(x\vee 1)$.

\par Thus the proof ends.
\end{proof}

\par Next we will see that under the conditions of the proposition above, the solution is smooth.

\begin{prop}\label{prop:B.solution.is.C.infinity.P}
If $f$ satisfies the conditions in Proposition \ref{prop:B.Equivalent.Conditions.P}, then
	\begin{equation}
	\label{B.solution.is.C.infinity.P}
		u(t,x)=\int\limits_{0}^{\infty}P_{t}^{\lambda}(x,y)f(y)d\mu_{\lambda}(y) \in C^{\infty}(\RR^{+}\times\RR^{+}).
	\end{equation}
\end{prop}
\begin{proof}
\par We need to show that for all $\alpha,\beta\ge 0$, $$\int\limits_{0}^{\infty} \left| \partial_{t}^{\alpha} \partial_{x}^{\beta} P_{t}^{\lambda}(x,y) \right| f(y) d\mu(y) < \infty.$$ Since the kernel $P_{t}^{\lambda}$ satisfies the Poisson equation for the Bessel operator, we only need to show that
\begin{align}\label{diff.under.int.sign.P}
\int\limits_{0}^{\infty} |\partial_{t}P_{t}^{\lambda}(x,y)| f(y) d\mu(y)<\infty.
\end{align}
\par To differentiate under the integral sign, from expression \ref{Poisson.expression} we compute
	\begin{align*}
	\partial_{t}P_{t}^{\lambda}(x,y) = & \frac{1}{4\sqrt{\pi}} (xy)^{-\lambda+\frac{1}{2}} \int\limits_{0}^{\infty} e^{-\frac{x^{2}+y^{2}+t^{2}}{4}z} I_{\lambda-\frac{1}{2}}\left(\frac{xy}{2}z\right) z^{\frac{1}{2}} du  \\
	 & +\frac{t(xy)^{-\lambda+\frac{1}{2}}}{4\sqrt{\pi}}  \int\limits_{0}^{\infty} \left(\frac{-2t}{4}z\right) e^{-\frac{t^{2}}{4}z} e^{-\frac{x^{2}+y^{2}}{4}z} I_{\lambda-\frac{1}{2}}\left(\frac{xy}{2}z\right) z^{\frac{1}{2}} dz  \\
	 = & A_{t}^{\lambda}(x,y) + B_{t}^{\lambda}(x,y).
\end{align*}
\par We have that $A_{t}^{\lambda}(x,y)=\frac{1}{t}P_{t}^{\lambda}(x,y)$ and
$$B_{t}^{\lambda}(x,y) \le  c \frac{t}{4\sqrt{\pi}} (xy)^{-\lambda+\frac{1}{2}} \int\limits_{0}^{\infty} e^{-\frac{t^{2}}{8}z}e^{-\frac{x^{2}+y^{2}}{4}z} I_{\lambda-\frac{1}{2}}\left(\frac{xy}{2}z\right) z^{\frac{1}{2}} dz = c P_{t/\sqrt{2}}^{\lambda}(x,y).$$
Since $f$ satisfies condition (i) of Proposition \ref{prop:B.Equivalent.Conditions.P}, \ref{diff.under.int.sign.P} follows.
\end{proof}

\par The last piece of the puzzle is the following proposition.

\begin{prop}\label{prop:B.solution.ae.limit.P}
If $f$ satisfies the conditions in Proposition \ref{prop:B.Equivalent.Conditions.P}, then
	\begin{equation}
	\label{B.solution.ae.limit.P}
		\lim\limits_{t\to 0^{+}} e^{-t\sqrt{\Delta_{\lambda}}}f(x)=f(x), \qquad \mbox{ a.e. } x>0.
	\end{equation}
\end{prop}

\begin{proof}

\par As usual, we prove the limit for $x\le n_{0}$, for all fixed $n_{0}\in\mathbb{NN}$. Indeed, let us split
$$f=f\chi_{|y|\le M} + f\chi_{|y|> M} =f_{1}+f_{2},$$ where $M>0$ will be chosen.

\par On one hand, if we consider $M>2n_{0}$, we have that from Lemma \ref{L.estimate.P}, since $2n_{0}>2 x$,
\begin{align*}
e^{-t\sqrt{\Delta_{\lambda}}}f_{2}(x) = & \int\limits_{|y|> M} P_{t}^{\lambda}(x,y)f(y)d\mu(y)  \\
\le & c(c_{0},\lambda) t \int\limits_{|y|> M} \phi^{\lambda}(y) \chi_{y>2x}(y) f(y) dy,
\end{align*}
hence, from Proposition \ref{prop:B.Equivalent.Conditions.P}, for $\epsilon>0$ there exists $M_{0}=M_{0}(\epsilon,n_{0})$ and $t_{0}>0$ such that if $M>M_{0}$, $t<t_{0}$ and $0<x\le n_{0}$, then  $e^{-t\sqrt{\Delta_{\lambda}}}f_{2}(x)<\epsilon$ . On the other hand, again we have that $e^{-t\sqrt{\Delta_{\lambda}}}$ defines a symmetric diffusion semigroup as in \cite{S} and $f_{1}\in L^{1}_{loc}(\RR^{+},d\mu)$, therefore $$\lim\limits_{t\to 0}e^{-t\sqrt{\Delta_{\lambda}}}f_{1}(x)=f_{1}(x)$$ for almost every $|x|\le n_{0}$. Thus the proposition follows.

\end{proof}

\par Finally we give the proof for Theorem \ref{thm:Lpv.Lpu.boundness.maximal.B.heat.operator}, item (ii).

\begin{proof}
\par Let us fix $1<p<\infty$, $\lambda >0$ and $a>0$. Let $v\in D_{p}^{Poisson}(\Delta_{\lambda})$. This means that $||v^{-\frac{1}{p}}\phi^{\lambda}||_{p'}<\infty$.

\par We need to show that the local maximal operator $P^{\lambda,\ast}_{a}$ defined in \ref{B.local.maximal.operator.P} maps $L^{p}(v)$ to $L^{p}(u)$ boundedly, for some weight $u$. Moreover, if $\sigma <1$ we need to find a weight $u$ such that, $||u^{-\frac{\sigma}{p}}\phi^{\lambda}||_{p'}<\infty$.

\par For $f\in L^{p}(v)$ and $M>1$ to be chosen later, let us use Lemma \ref{L.estimate.P} to obtain the split
\begin{align*}
y^{2\lambda}P_{a}^{\lambda,\ast}f(x) \le & c_{\lambda,M} (x\wedge1)^{-2\lambda} \sup\limits_{0<t<a} \int\limits_{y\le M(x\vee 1)} P_{t}^{\triangle}(x-y) (y\wedge1)^{2\lambda} f(y)dy  \\
& + c_{\lambda,M} a  \int\limits_{y> M(x\vee 1)} \phi^{\lambda}(y) f(y) dy = Af(x) + Bf(x).
\end{align*}

\par Let us consider first $Af(x)$. Using one more time the standard argument of slicing the integral into dyadic shells we get that
\begin{align*}
Af(x) \le & c_{\lambda,M} (x \wedge 1)^{-2\lambda}  \mathcal{M}_{M}^{loc}(\tilde{f}),
\end{align*}
where $\tilde{f}= (y\wedge1)^{2\lambda}f(y) $ and $\mathcal{M}_{M}^{loc}$ is, as before,  the local maximal function considered in \ref{local.maximal.operator}, extended to $|y|\le M(|x|\vee 1)$, namely $$\mathcal{M}_{M}^{loc}f(x)=\sup_{r>0}\frac{1}{|B_{r}|} \int\limits_{B_{r}(x)}f(y)\chi_{|y|\le M(x\vee 1)}(y)dy.$$ Then, by using item (ii) of Theorem \ref{thm:Lpv.Lpu.boundness.local.maximal}, we get that there exists a weight $\tilde{u}_2$ such that
$$||\mathcal{M}_{M}^{loc}(\widetilde{f})||_{L^{p}(\tu_2)}\le C ||\widetilde{f}||_{L^{p}(\widetilde{v})}= C ||f||_{L^{p}(v)}$$ and
\begin{equation}\label{Eu}
\left|\left|\tu_2^{-\frac{\sigma}{p}}(x\vee1)^{-2}\right|\right|_{p'}<\infty \qquad \mbox{for $\sigma <1,$}
\end{equation}
provided that
\begin{equation}\label{condition v2}
\left|\left|\tv^{-\frac{1}{p}}(y\vee1)^{-2}\right|\right|_{p'}<\infty,
\end{equation}
where $\tv= (y\wedge1)^{-2\lambda p}$. Since $v\in D_{p}^{Poisson}(\Delta_{\lambda})$, (\ref{condition v2}) follows. This implies that Theorem \ref{thm:Lpv.Lpu.boundness.local.maximal}  holds. Now taking $u_{2}(x)=\tilde{u}_{2}(x\wedge 1)^{2\lambda p}$, it follows that $$||A(f)||_{L^{p}(u_{2})} \le c ||f||_{L^{p}(v)}$$ and $$||u_{2}^{-\frac{\sigma}{p}}\phi^{\lambda}||_{p'} \le \left|\left|\tilde{u}_{2}^{-\frac{\sigma}{p}}(x\wedge 1)^{2\lambda(1-\sigma)} (x\vee1)^{-2}\right|\right|_{p'}<\infty.$$

\par For $Bf(x)$ let us apply H\"{o}lder inequality to obtain
\begin{align*}
Bf(x) \le & c_{\lambda,M} a \int\limits_{y>M(x\vee 1)} \phi^{\lambda}v^{-\frac{1}{p}}v^{\frac{1}{p}}f(y)dy \\
\le & c_{\lambda,M} a ||f||_{L^{p}(v)} ||v^{-\frac{1}{p}}\phi^{\lambda}||_{p'}.
\end{align*}
Then, $$||B(f)||_{L^{p}(u_{1})} \le c ||f||_{L^{p}(v)}$$ if we choose $$u_{1}(x)=\frac{1}{(1+x)^{2}}.$$

\par The Theorem follows by considering a weight $$u(x)=\min\{u_{1}(x),u_{2}(x)\}.$$

\end{proof}

\subsection*{Acknowledgments}
The author thanks Beatriz Viviani for her kind guidance and advice in the development of this work.

\end{document}